\documentclass[12pt]{article}
\usepackage{amsmath}
\usepackage{amssymb}
\usepackage{amsthm}
\usepackage{mathrsfs}
\usepackage[pdftex]{graphicx}

\newtheorem{thm}{Theorem}
\newtheorem{prop}{Proposition}
\newtheorem{defn}{Definition}

\newtheorem{rmk}{Remark}

\begin{document}
\title{{\bf The Monge-Amp\`{e}re Equation}}
\author{Connor Mooney}
\date{2018}
\maketitle
\tableofcontents

\newpage
\section{Introduction}
In this survey we discuss the basic properties of Alexandrov solutions to the Monge-Amp\`{e}re equation. We then discuss the interior and boundary regularity for Alexandrov solutions to $\det D^2u = 1$. At the end we list some recent advances
and open questions.

\vspace{3mm}

Our choice of topics reflects the state of the subject as of $\sim 1990$. Since many of these topics (as well as more modern ones) have detailed expositions
elsewhere, we have endeavored to highlight only the key ideas and to cite appropriate references. However, we also decided to include some important results which it seems were not revisited by the PDE
community in recent times. These include Calabi's interior $C^3$ estimate for solutions to $\det D^2u = 1$ (\cite{Cal}, 1958), and the approaches of Cheng-Yau (\cite{CY}, 1977) and Lions (\cite{L}, 1983)
to obtain classical solutions to the Dirichlet problem.

\vspace{3mm}

The survey is based on two mini-courses given by the author in May 2018. One was for ``Advanced Lectures in Nonlinear Analysis" at l'Universit\`{a} degli Studi di Torino, and the other for the Oxford PDE CDT.
It is my pleasure to thank Paolo Caldiroli, Francesca Colasuonno, and Susanna Terracini for their kind invitation and hospitality in Torino. I am very grateful
to John Ball for the invitation to Oxford. Finally, I would like to thank Alessio Figalli for comments on a preliminary draft.
This work was partially supported by NSF grant DMS-1501152 and the ERC grant ``Regularity and Stability in PDEs.''

\newpage
\section{Motivation}
The Monge-Amp\`{e}re equation
$$\det D^2u = f(x,\,u,\,\nabla u)$$
for a convex function $u$ on $\mathbb{R}^n$ arises in several interesting applications. In this section we list a few of them.

\subsection{Prescribed Gauss Curvature}
The Gauss curvature $K(x)$ of the graph of a function $u$ on $\mathbb{R}^n$ at $(x,\,u(x))$ is given by
$$\det D^2u = K(x)(1 + |\nabla u|^2)^{\frac{n+2}{2}}.$$
It is a good exercise to derive this formula.

\subsection{Optimal Transport}
Given probability densities $f,\,g$ supported on domains $\Omega_f,\, \Omega_g$ in $\mathbb{R}^n$, the optimal transport problem asks to minimize the transport cost
$$J(T) = \int_{\Omega_f} |T(x) - x|^2f(x)\,dx$$
over measure-preserving maps $T: \Omega_f \rightarrow \Omega_g$ (that is, $f(x)\,dx = g(T(x))\det DT(x)dx$). An important theorem of Brenier says that the optimal map exists,
and is given by the gradient of a convex function $u$ on $\Omega_f$ \cite{Br}. The measure-preserving condition implies
$$\det D^2u = \frac{f(x)}{g(\nabla u(x))},$$
in a certain weak sense.

\subsection{Fluid Dynamics}
Large-scale fluid flows in $\mathbb{R}^2$ are modeled by a system of evolution equations for a probability density $\rho(x,\,t)$ and a function $u(x,\,t)$ that is convex in $x$ for all $t$. The system is
$$\begin{cases}
\partial_t\rho + (x- \nabla u)^{\perp} \cdot \nabla \rho = 0, \\
\det D^2u = \rho.
\end{cases} $$
Here $w^{\perp}$ denotes the counter-clockwise rotation of $w$ by $\frac{\pi}{2}$.
This can be viewed as a fully nonlinear version of the incompressible Euler equations in $2D$, where the Monge-Amp\`{e}re operator replaces
the Laplace operator.

\newpage
\section{Weak Solutions}
In this section we introduce a useful notion of weak solution based on the idea of polyhedral approximations.
We then solve the Dirichlet problem on bounded domains. Detailed expositions of these topics can be found in work of Cheng-Yau \cite{CY}, and in the books of Gutierrez \cite{Gut}
and Figalli \cite{F}.


\subsection{Alexandrov Solutions}
If $v$ is a $C^2$ convex function on $\mathbb{R}^n$, then the area formula gives
$$\int_{\Omega} \det D^2v \,dx = |\nabla v(\Omega)|.$$
For an arbitrary convex function $v$ on a domain $\Omega \subset \mathbb{R}^n$ and $E \subset \Omega$ we define
$$Mv(E) = |\partial v(E)|,$$
where $\partial v(E)$ is the set of slopes of supporting hyperplanes to the graph of $v$ (the sub-gradients of $v$) over points in $E$. We have (\cite{F}, Theorem $2.3$):
\begin{prop}\label{MAMeasure}
$Mv$ is a Borel measure on $\Omega$.
\end{prop}
\noindent 
We call $Mv$ the Monge-Amp\`{e}re measure of $v$. 

It is easy to check that if $v \in C^2$, then $Mv = \det D^2v\,dx$. A more interesting example is the polyhedral graph 
$$v = \max_{1 \leq i \leq 3}\{p_i \cdot x\}$$ 
over $\mathbb{R}^2$. The set $\partial v(0)$ is the (closed) triangle with vertices $\{p_i\}$. The sub-gradients
of the ``edges" of the graph are the segments joining $p_i$, and the sub-gradients of the ``faces" are $p_i$. Thus, $Mv$ is a Dirac mass at $0$ with weight given by the area
of the triangle. 

\begin{defn}
Let $\mu$ be a Borel measure on a domain $\Omega \subset \mathbb{R}^n$. We say that a convex function $u$ on $\Omega$ is an Alexandrov solution of $\det D^2u = \mu$ if $Mu = \mu$.
\end{defn}

The key fact that is used to prove Proposition \ref{MAMeasure}, and is essential for many other parts of the theory, is (\cite{F}, Lemma $A.30$):

\begin{prop}\label{KeyFact}
Let $v$ be a convex function on a domain $\Omega \subset \mathbb{R}^n$. Then 
$$|\{p \in \mathbb{R}^n: p \in \partial v(x) \cap \partial v(y) \text{ for some } x \neq y \in \Omega\}| = 0.$$
\end{prop}
\noindent In particular, if $E_1,\,E_2 \subset \Omega$ and $E_1 \cap E_2 = \emptyset$, then $|\partial v(E_1) \cap \partial v(E_2)| = 0$.
To understand heuristically why the latter is true, consider the case that $\Omega = \mathbb{R}^n$ and $E_i$ are compact, with $E_1 \subset \{x_n < 0\}$ and $E_2 \subset \{x_n > 0\}$. 
If $0 \in \partial v(E_1),$ then by convexity (monotonicity of subgradients) we have that $\partial v(E_2)$ is missing a thin cone around the negative $x_n$-axis.
Thus, $0$ is not a Lebesgue point of $\partial v(E_2)$.


\subsection{Maximum Principle and Compactness}
Alexandrov solutions are useful because they satisfy a maximum principle and have good compactness properties.

\vspace{3mm}

We first observe that if $u$ and $v$ are convex on a bounded domain $\Omega$, with $u = v$ on $\partial \Omega$ and $u \leq v$ in $\Omega$, then
$$\partial v(\Omega) \subset \partial u(\Omega).$$
This is a simple consequence of convexity. From this observation one concludes the comparison principle (\cite{F}, Theorem $2.10$):
\begin{prop}\label{ComparisonPrinciple}
Assume $u$ and $v$ are convex on a bounded domain $\Omega$, with $u = v$ on $\partial \Omega$. If $Mu \geq Mv$ in $\Omega$, then $u \leq v$ in $\Omega$.
\end{prop}
\noindent Another important consequence is the Alexandrov maximum principle (\cite{F}, Theorem $2.8$):
\begin{prop}\label{AlexandrovMaxPrinciple}
 If $u$ is convex on a bounded convex domain $\Omega$ and $u|_{\partial \Omega} = 0$, then
 $$|u(x)| \leq C(n,\,\text{diam}(\Omega), Mu(\Omega))\text{dist.}(x,\,\partial \Omega)^{1/n}.$$
\end{prop}
\noindent This says that functions with bounded Monge-Amp\`{e}re mass have a $C^{1/n}$ modulus of continuity near the boundary of a sub level set, that depends only on rough geometric properties of this set.
The proof is to compare $u$ with the cone $v$ with vertex $(x,\,u(x))$ passing through $(\partial \Omega,\,0)$. Indeed, $\partial v(x)$ contains a point of size $|u(x)|/\text{dist.}(x,\,\partial\Omega)$
and a ball of radius $|u(x)|/\text{diam.}(\Omega)$. Since $\partial v(x)$ is convex we conclude that 
$$Mv(x) \geq c(n)|u(x)|^n/[\text{dist.}(x,\,\partial\Omega)\text{diam.}^{n-1}(\Omega)].$$

The other important property of Alexandrov solutions is closedness under uniform convergence. That is:
\begin{prop}\label{Closedness}
If $u_k$ converge uniformly to $u$ in $\Omega \subset \mathbb{R}^n$, then $Mu_k$ converges weakly to $Mu$ in $\Omega$.
\end{prop}
\noindent Roughly, if $p \in \partial u(x_0)$ and the supporting plane of slope $p$ touches only over $x_0$, then it is geometrically clear that $p \in \partial u_k(x_k)$ with $x_k \rightarrow x_0$. By the key fact
Proposition \ref{KeyFact} we may ignore the remaining sub-gradients. For a detailed proof see \cite{F}, Proposition $2.6$.

As a consequence of Propositions \ref{AlexandrovMaxPrinciple} and \ref{Closedness} we have the compactness of solutions with fixed linear boundary data:

\begin{prop}\label{Compactness}
For a bounded convex domain $\Omega$, the collection of functions $$\mathcal{A} = \{v \text{ convex on } \Omega,\, v|_{\partial \Omega} = 0, \, Mv(\Omega) \leq C_0\}$$ is compact. That is, any sequence in $\mathcal{A}$ has a uniformly convergent subsequence whose Monge-Amp\`{e}re measures converge weakly to that of the limit.
\end{prop}


\subsection{Dirichlet Problem}
We conclude the section by discussing the Dirichlet problem.

\begin{thm}
Let $\Omega \subset \mathbb{R}^n$ be a bounded strictly convex domain, $\mu$ a bounded Borel measure on $\Omega$, and $\varphi \in C(\partial \Omega)$. Then there exists a unique Alexandrov solution in $C(\overline{\Omega})$ to
the Dirichlet problem
$$\begin{cases}
\det D^2u = \mu \text{ in } \Omega, \\
u|_{\partial \Omega} = \varphi.
\end{cases}
$$
\end{thm}

\begin{proof}[{\bf Sketch of Proof:}]
Uniqueness follows from the comparison principle Proposition \ref{ComparisonPrinciple}. 
For existence, to emphasize ideas we treat the case that $\varphi = 0$ and $\Omega$ is a polyhedron. We note that $\mu$ is weakly approximated by finite sums of Dirac masses, $\sum_{i = 1}^N \alpha_i \delta_{x_i}$. 
By the compactness result Proposition \ref{Compactness}, it suffices to consider this case.

Let $\mathcal{F}$ be the family of convex polyhedral graphs $P$ in $\mathbb{R}^{n+1}$ that contain $(\partial \Omega,\, 0) \subset \mathbb{R}^n \times \mathbb{R}$, with remaining vertices that project to a subset of $\{x_i\}_{i = 1}^N$. Let 
$\mathcal{F}' \subset \mathcal{F}$ consist of those $P$ satisfying $MP \leq \sum_{i = 1}^N \alpha_i \delta_{x_i}$. The family $\mathcal{F}'$ is non-empty ($0$ is a trivial example) and compact by Proposition \ref{Compactness}.

For $P \in \mathcal{F}$ we let $\phi(P) = \sum_{i = 1}^N P(x_i)$. The functional $\phi$ is bounded below on $\mathcal{F}'$ by the Alexandrov maximum principle. By compactness
there exists a minimizer $u$ of $\phi$ in $\mathcal{F}'$. We claim that $u$ solves $Mu = \sum_{i = 1}^N \alpha_i \delta_{x_i}$. If not, then after re-labeling we have $Mu(\{x_1\}) < \alpha_1$. By moving the vertex
$(x_1,\, u(x_1))$ a tiny bit downwards and taking the convex hull of this point with the remaining vertices, we obtain another function in $\mathcal{F}'$ that is smaller than $u$, a contradiction.

\vspace{3mm}

The case that $\Omega$ is a polyhedron and $\varphi$ is affine on each face of $\partial \Omega$ is treated similarly, with $\mathcal{F},\, \mathcal{F}'$ consisting of convex polyhedral graphs with vertices over $\{x_i\}_{i = 1}^N$ and
$\varphi$ as boundary data. To show that $\mathcal{F}'$ is non-empty, use instead the convex hull of the graph of $\varphi$ in $\mathbb{R}^{n+1}$.

\vspace{3mm}

Finally, for the general case, approximate $\Omega$ with the convex hulls of finite subsets of $\partial \Omega$ with finer and finer mesh, approximate $\varphi$ by data which are affine on the faces
of these polyhedra, solve these problems, and take a limit. We refer the reader to \cite{CY} for details.
\end{proof}

\begin{rmk}
When $\varphi$ is linear, we don't require strict convexity of $\partial \Omega$.
The strict convexity is necessary for general $\varphi$ since no convex function can continuously attain e.g. the boundary data $-|x|^2$ when $\partial \Omega$ has flat pieces.

The strict convexity of $\partial \Omega$ is used in the last step. It guarantees that for any subset $\{y_i\}_{i = 1}^M$ of $\partial \Omega$, each $y_k$ is a vertex of the convex hull of $\{y_i\}_{i = 1}^M$.

Closely related is the fact that when $\partial \Omega$ is strictly convex, the convex envelope of the graph of $\varphi$ (that is, the supremum of linear functions beneath the graph) continuously achieves the boundary data,
and has $0$ Monge-Amp\`{e}re measure. (This is a good exercise. For the second part, recall the key fact Proposition \ref{KeyFact}).
\end{rmk}

\newpage
\section{Interior Regularity}
In this section we discuss the interior regularity problem for the important case $f = 1$.


\subsection{Structure of the Equation}
We start by listing some important structural properties of the equation
\begin{equation}\label{Equation}
\det D^2u = 1.
\end{equation}

One can view this equation as a differential inclusion which says that $D^2u$ lies in the hypersurface of positive, determinant $1$ matrices. A useful playground
for investigating matrix geometry is $\text{Sym}_{2 \times 2} \cong \mathbb{R}^3$, by identifying $I$ with $e_3$, and the traceless matrices with the subspace $\{x_3 = 0\}$. One can check that
the surfaces of determinant $c_0$ are the hyperboloid sheets
$$x_3^2 = x_1^2 + x_2^2 + c_0.$$
In particular, the cone of positive matrices is $\{x_3^2 > x_1^2 + x_2^2\}$, and the surfaces of constant determinant in this cone are convex.

\vspace{3mm}

If we parametrize the level surfaces of $\det$ correctly, we obtain a concave function on the positive matrices. For $M,\, N \in \text{Sym}_{n \times n}$ with $M > 0$, it is a good exercise to derive the expansion
\begin{equation}\label{logdetExpansion}
\log\det(M + \epsilon N) = \log\det(M) + \epsilon M^{ij}N_{ij} - \frac{\epsilon^2}{2} M^{ik}M^{jl}N_{ij}N_{kl} + O(\epsilon^3).
\end{equation}
Here $M^{ij} = (M^{-1})_{ij}$, and repeated indices are summed. The first-order term is positive when $N \geq 0$, and the second-order term is negative for any $N$.
We conclude that $\log\det$ is elliptic and concave on the the positive symmetric matrices. Furthermore, it is uniformly elliptic when restricted to bounded regions on the surface of positive, determinant $1$ matrices. 
To be precise, for any $C_0 > 1$ there exists a concave, uniformly elliptic extension of $\log\det$ from $\{M > 0,\, \det M = 1,\, |M| < C_0\}$ to $\text{Sym}_{n \times n}$. (The ellipticity constants depend on $C_0$).
Thus, solutions to (\ref{Equation}) solve a concave, uniformly elliptic equation provided $|D^2u|$ is bounded.

\vspace{3mm}

The landmark result for such equations, due to Evans \cite{E} and Krylov \cite{Kr}, is:
\begin{thm}\label{EvansKrylov}
Assume $F(D^2w) = 0$ in $B_1 \subset \mathbb{R}^n$, where $F$ is uniformly elliptic and concave. Then
$$\|w\|_{C^{2,\,\alpha}(B_{1/2})} \leq C(n,\, F)\|w\|_{L^{\infty}(B_1)}.$$
\end{thm} 
\noindent The dependence of $C$ on $F$ is through the ellipticity constants of $F$.

\begin{rmk}
The Evans-Krylov theorem can be understood heuristically as follows.
The concavity of $F$ guarantees that the pure second derivatives of a solution
are sub-solutions to the linearized equation. They also behave like super-solutions by ellipticity (a pure second derivative is a ``negative combination" of others). 
Finally, a fundamental result of Krylov and Safonov (see \cite{CC}) says that solutions to linear non-divergence uniformly elliptic equations with bounded measurable coefficients are $C^{\alpha}$. Good references for the Evans-Krylov theory include
\cite{GT}, Chapter $17$ and \cite{CC}, Chapter $6$.
\end{rmk}

Let $e$ be a unit vector. Applying the formula (\ref{logdetExpansion}) to $D^2(u + \epsilon e) = D^2u(x) + \epsilon D^2u_{e}(x) + \frac{\epsilon^2}{2}D^2u_{ee} +  O(\epsilon^3)$
we obtain the once- and twice- differentiated equations
\begin{equation}\label{DifferentiatedEquations}
u^{ij}u_{eij} = 0, \quad u^{ij}u_{eeij} = u^{ik}u^{jl}u_{eij}u_{ekl}.
\end{equation}
\noindent Here $u^{ij}$ are the components of $(D^2u)^{-1}$, and subscripts denote derivatives. By combining Theorem \ref{EvansKrylov} with Schauder theory for the once-differentiated equation, we conclude that solutions to (\ref{Equation}) satisfy
$$\|u\|_{C^k(B_{1/2})} \leq C(n,\,k,\, \|u\|_{C^{1,\,1}(B_1)})$$
for all $k \geq 0$. Thus, the key to regularity is to bound $|D^2u|$.

\vspace{3mm}

The last important structural property of (\ref{Equation}) is the affine invariance:
\begin{equation}\label{AffineInvar}
 \tilde{u}(x) = |\det A|^{-2/n}u(Ax)
\end{equation}
also solves (\ref{Equation}), for any invertible affine transformation $A$. This distinguishes the character of the Monge-Amp\`{e}re equation from e.g. the minimal surface or Laplace equations.
Since for any $x_0$ we have $D^2\tilde{u}(x_0) = I$ after some affine transformation of determinant $1$, the equation (\ref{Equation}) can be viewed as the affine-invariant Laplace equation.

\subsection{Calabi's $C^3$ Estimate}
Remarkably, Calabi reduced the regularity problem for (\ref{Equation}) to second derivative estimates well before the Evans-Krylov breakthrough. He proved the following $C^3$ estimate \cite{Cal}:
\begin{thm}\label{Calabi}
Assume that $u \in C^5(B_1)$ solves (\ref{Equation}). Then
$$\|u\|_{C^3(B_{1/2})} \leq C(n,\, \|u\|_{C^2(B_1)}).$$
\end{thm}
\noindent Calabi's approach is to derive a second-order differential inequality for the quantity
$$R := u^{kp}u^{lq}u^{mr}u_{klm}u_{pqr}.$$
This quantity is in fact the scalar curvature of the metric $g = u_{ij}$. He showed:
\begin{prop}\label{DifferentialCalabi}
The quantity $R$ satisfies
\begin{equation}\label{GoodEquation}
u^{ij}R_{ij} \geq \frac{1}{2n}R^2.
\end{equation}
\end{prop}

\noindent The quadratic nonlinearity on the right side of inequality (\ref{GoodEquation}) is powerful. ODE intuition indicates that if $R(0)$ is very large, then it must blow up close to $0$, where 
closeness is measured with respect to the metric $g = u_{ij}$. Theorem \ref{Calabi} follows from Proposition \ref{DifferentialCalabi} by observing that if $|D^2u| < C_0I$ in $B_1$, then
$K(1-2|x|^2)^{-2}$ is a super-solution to (\ref{GoodEquation}) for $K$ large depending on $n$ and $C_0$, and blows up on $\partial B_{1/\sqrt{2}}$.

\vspace{3mm}

We now prove Proposition \ref{DifferentialCalabi}.
\begin{proof}[{\bf Proof of Proposition \ref{DifferentialCalabi}:}]
It suffices to derive the inequality at $x = 0$. To simplify computations, we observe that under affine rescalings $u \rightarrow u(Ax)$, the quantity $R$ remains invariant (that is, $R$ becomes $R(Ax)$). 
We may thus assume that $D^2u(0) = I$.

\vspace{2mm}

Since $R$ is a third-order quantity, we differentiate the equation (\ref{Equation}) three times. Evaluating at $x = 0$ we obtain
\begin{equation}\label{DifferentiatedEquation}
u_{kii} = 0, \quad u_{klii} = u_{kij}u_{lij}, \quad u_{klmii} = u_{ij[kl}u_{m]ij} - 2 u_{kij}u_{pil}u_{pjm}.
\end{equation}
Here repeated indices are summed, and the brackets indicate a sum of three terms obtained by cyclically permuting the indices.

\vspace{2mm}

Now we compute $R_{ii}(0)$. When both derivatives hit the same third-order term in $R$, we get $2u_{klmii}u_{klm}$. Using the thee times-differentiated equation, we get
$$(I) := 6u_{ijkl}u_{mij}u_{mkl} - 4u_{kij}u_{klm}u_{pil}u_{pjm}.$$
The first term comes from the cyclic sum, and the symmetry of $u_{klm}$.
Since we have no information on the sign of this expression, we view this as a ``bad term."

When the derivatives hit different third-order terms in $R$ we get
$$(II) := 2u_{ijkl}^2 = 2|D^4u|^2,$$
which is a ``good term" coming from the fact that $R$ is a quadratic function of third derivatives.

We now consider the case that both derivatives hit the same $D^2u^{-1}$ term. Using that $\partial_i^2(u^{kp}) = -u_{kpii} + 2u_{iak}u_{iap}$ and applying the twice-differentiated equation we obtain
$$(III) := 3u_{kij}u_{pij}u_{klm}u_{plm} := 3A \geq 0.$$
This is a ``good term" coming from the structure of the equation.

When the derivatives hit different $D^2u^{-1}$ terms in $R$ we get
$$(IV) := 6u_{ikp}u_{ilq}u_{mkl}u_{mpq} := 6B.$$
Note that this has the same form as the second term in $(I)$.

Finally, when one derivative hits a $D^2u^{-1}$ term and the other a third-order term in $R$, we get
$$(V) := -12u_{iklm}u_{pik}u_{plm}.$$
Note that this has the same form as the first term in $(I)$.

Summing $(I),\, (II),\,(III),\,(IV)$ and $(V)$ we obtain
$$\Delta R(0) = -6u_{ijkl}u_{pij}u_{pkl} + 2|D^4u|^2 + 3A + 2B.$$
Since the first term is a product of $D^4u$ and $(D^3u)^2$ and the remaining terms are quadratic in these quantities, we can hope to absorb it.
To that end we observe that
$$2\left(u_{ijkl} \pm \frac{1}{2}u_{p[ij}u_{k]lp}\right)^2 = 2|D^4u|^2 \pm 6u_{ijkl}u_{pij}u_{pkl} + \frac{3}{2}A + 3 B \geq 0.$$
Using this inequality in the equation for $R$ at $0$ gives
$$\Delta R(0) \geq \frac{3}{2}A - B \geq \frac{1}{2}A + (A-B).$$
Observe that the quantity $A$ can be written
$$A = \sum_{k,p} \left(\sum_{i,j} u_{kij}u_{pij}\right)^2 = \sum_{i,j,l,m} \left(\sum_p u_{pij}u_{plm}\right)^2.$$
In the first way of writing $A$, we consider only the case $k=p$ and use Cauchy-Schwarz to get
$$A \geq \frac{1}{n}R^2.$$
The second way of writing $A$ makes it clear that
$$6(A - B) = \sum_{i,j,l,m} \left(\sum_k u_{kij}u_{klm} + u_{kjl}u_{kim} -2u_{kli}u_{kjm}\right)^2 \geq 0.$$
We conclude that 
$$\Delta R(0) \geq \frac{1}{2}A \geq \frac{1}{2n}R^2,$$
completing the proof.
\end{proof}

\begin{rmk}
Inequalities of the type (\ref{GoodEquation}) play an important role in elliptic PDE and geometric analysis. The model example is:
if $u$ is a positive harmonic function, then the quantity $w := |\nabla(\log u)|^2$ satisfies $\Delta w \geq \frac{2}{n}w^2 + b \cdot \nabla w$, with $|b| \sim |w|^{1/2}$. This implies the Harnack inequality for harmonic functions.
It is a good exercise to derive this inequality. 

A similar inequality known as the ``Li-Yau differential Harnack inequality" holds for caloric functions, and it is useful in the study of geometric flows.
\end{rmk}


\subsection{Pogorelov's $C^2$ Estimate}
An important breakthrough in the theory for (\ref{Equation}) is the following interior $C^2$ estimate of Pogorelov \cite{Pog}:
\begin{thm}
Let $\Omega \subset \mathbb{R}^n$ be a bounded convex domain, and assume that $u \in C^4(\Omega) \cap C^2(\overline{\Omega})$ solves
$$\begin{cases} \det D^2u = f(x_2,\,...,\,x_n) > 0 \quad \text{ in } \Omega, \\
u|_{\partial \Omega} = 0.
\end{cases}
$$
Then
$$|u|u_{11} \leq C\left(n,\,\sup_{\Omega}u_1\right).$$
\end{thm}

\noindent When $f = 1$ the Pogorelov estimate bounds $D^2u(x)$ in terms of the distance from $x$ to $\partial \Omega$ and rough geometric properties (volume, diameter) of $\Omega$.

\begin{rmk}
The Pogorelov estimate implies a Liouville theorem for the Monge-Amp\`{e}re equation: the only global smooth ($C^4$) solutions to (\ref{Equation}) are quadratic polynomials. 
Liouville theorems are closely connected to regularity, and often direct connections can be made by blowup procedures.
\end{rmk}

\begin{rmk}
The Pogorelov estimate suggests that strictly convex solutions to (\ref{Equation}) are smooth. This was justified independently by Cheng-Yau \cite{CY} and Lions \cite{L}. We discuss their
results later in this section.
\end{rmk}

\begin{proof}[{\bf Proof of Pogorelov Estimate}:]
We apply the maximum principle to the quantity
$$Q := \log u_{11} + \log |u| + \frac{1}{2}u_1^2.$$
The first and last terms are ``good terms" that are sub-solutions to the linearized equation, and thus don't take interior maxima. The second term acts as a cutoff, guaranteeing that $Q$ attains
an interior maximum, say at $0$.

We may assume by affine invariance that $D^2u(0)$ is diagonal. Indeed, consider the affine change of coordinates
$$u \rightarrow u\left(x_1 - \frac{u_{12}}{u_{11}}(0)x_2 - ... - \frac{u_{1n}}{u_{11}}(0)x_n,\, x_2,\, ...,\,x_n\right).$$
This transformation preserves the $x_1$-derivatives of $u$ (hence $Q$), and the equation. Furthermore the mixed derivatives $u_{1k}(0)$ become $0$ for $k \geq 2$.
We can then rotate in the $x_2,\,...,\,x_n$ variables to make $D^2u(0)$ diagonal. 
(Roughly, we use the invariance of the equation under certain shearing transformations to align the axes of the ``$D^2u(0)$-spheres" with the coordinate directions, without affecting the derivatives in the direction of interest.)

The first and second derivatives of the equation in the $e_1$ direction at $x = 0$ are
\begin{equation}\label{PogDifferentiatedEquations}
\frac{u_{1ii}}{u_{ii}} = 0, \quad \frac{u_{11ii}}{u_{ii}} = \frac{u_{1ij}^2}{u_{ii}u_{jj}}.
\end{equation}

The condition that $\nabla Q(0) = 0$ is
$$\frac{u_{11i}}{u_{11}} + \frac{u_i}{u} + u_1u_{1i} = 0, \quad 1 \leq i \leq n.$$
Note that the last term vanishes when $i \geq 2$, since $D^2u(0)$ is diagonal.

Finally, the condition that $\frac{Q_{ii}}{u_{ii}}(0) \leq 0$ can be written
$$\frac{u_{11ii}}{u_{11}u_{ii}} - \frac{u_{11i}^2}{u_{11}^2u_{ii}} + \frac{n}{u} - \frac{u_1^2}{u^2u_{11}} - \left.\frac{u_i^2}{u^2u_{ii}}\right\vert_{i = 2}^n + u_{11} + u_1\frac{u_{1ii}}{u_{ii}} \leq 0.$$
The last term vanishes by the linearized equation (the first equation in (\ref{PogDifferentiatedEquations})). Using the twice-differentiated equation for the first term and the condition $\nabla Q(0) = 0$ for the fifth term and grouping these we get
$$\frac{1}{u_{11}}\left(\frac{u_{1ij}^2}{u_{ii}u_{jj}} - \frac{u_{11i}^2}{u_{11}u_{ii}} - \left.\frac{u_{11j}^2}{u_{11}u_{jj}}\right\vert_{j = 2}^n \right) + \frac{n}{u} - \frac{u_1^2}{u^2u_{11}} + u_{11} \leq 0.$$
The first term is positive, so we conclude that
$$(|u|u_{11})^2 - n|u|u_{11} \leq u_1^2$$
at $x = 0$. Since the left side is a quadratic polynomial in $|u|u_{11}$, this gives the desired inequality at $x = 0$.
At a general point $x \in \Omega$ we have 
$$|u|u_{11}(x) \leq e^{Q(x)} \leq e^{Q(0)} \leq C\left(n,\,\sup_{\Omega}u_1\right),$$
completing the proof.
\end{proof}

\begin{rmk}
For general right hand side $f(x) > 0$, the Pogorelov computation gives interior $C^2$ estimates that depend also on $\|\log f\|_{C^2(\Omega)}$.
\end{rmk}

\subsection{The Pogorelov Example}
Pogorelov constructed singular solutions to (\ref{Equation}) in dimension $n \geq 3$ that have line segments in their graphs. We discuss them here.

\begin{rmk}
It is a classical fact solutions to (\ref{Equation}) are strictly convex in the case $n=2$ \cite{A}. 
See e.g. \cite{M1}, Lemma $2.3$ for a simple proof.
\end{rmk}

\vspace{3mm}

Write $x = (x',\,x_n) \in \mathbb{R}^n$. The Pogorelov example has the form
\begin{equation}\label{PogorelovExample}
u(x) = |x'|^{2-2/n}h(x_n).
\end{equation}
Note that $u$ is invariant under rotations around the $x_n$-axis and under the rescalings 
$$u \rightarrow \frac{1}{\lambda^{2-2/n}}u(\lambda x',\, x_n),$$ 
which preserve (\ref{Equation}). Away from $\{|x'| = 0\}$, the equation (\ref{Equation}) for $u$ is equivalent to the ODE
$$h^{n-2}\left(h\,h'' - \frac{2n-2}{n-2}h'^2\right) = c(n) > 0$$
for $h$. There exists a convex, positive, even solution to this ODE around $0$ with the initial conditions $h(0) = 1$ and  $h'(0) = 0$. It is a good exercise to show that $h$ blows up in finite time $\rho_n$.
With this choice of $h$, $u$ is convex and solves $(\ref{Equation})$ in the slab $\{|x_n| < \rho_n\}$ away from $\{|x'| = 0\}$ (where the graph has a line segment).
Another good exercise is to verify that $u$ is an Alexandrov solution to (\ref{Equation}) in $\{|x_n| < \rho_n\}$.

\vspace{3mm}

The Pogorelov example is not a classical solution. Roughly, $u_{nn}$ goes to $0$ near the $x_n$ axis, and the other second derivatives go to $\infty$ at just the right rate for the product to be $1$.
We note that $u$ is $C^{1,\,1-2/n}$, and $W^{2,\,p}$ for $p < n(n-1)/2$.

\begin{rmk}
The Pogorelov example shows that that there is no pure interior regularity for (\ref{Equation}) in dimensions $n \geq 3$ (in contrast with the Laplace and minimal surface
equations). This is closely related to the affine invariance.
\end{rmk}

\begin{rmk}
There are singular solutions to (\ref{Equation}) in dimension $n \geq 3$ that are merely Lipschitz, and still others that degenerate on higher-dimensional subspaces (up to any dimension
strictly smaller than $n/2$). They can all be viewed
as generalizations of the Pogorelov example. See e.g. \cite{M3}, Chapter $2$ for a detailed discussion of these examples.
\end{rmk}

\begin{rmk}
The Pogorelov example has the ``best possible regularity" for a singular solution. More specifically, a solution $u$ to (\ref{Equation}) is strictly convex
if either $u \in C^{1,\,\beta}$ for $\beta > 1-2/n$ (see \cite{U}) or $u \in W^{2,\,p}$ with $p \geq n(n-1)/2$ (see \cite{CM}).
\end{rmk}


\subsection{The Contributions of Cheng-Yau and Lions}
When $\partial \Omega$ is smooth and uniformly convex, Cheng-Yau \cite{CY} and Lions \cite{L} obtained solutions in $C^{\infty}(\Omega) \cap C(\overline{\Omega})$ to the Dirichlet problem
$$\det D^2u = 1 \text{ in } \Omega, \quad u|_{\partial \Omega} = 0$$
using different techniques. By approximation, this result implies that any strictly convex Alexandrov solution to (\ref{Equation}) is smooth. 
Both methods rely on the Pogorelov interior $C^2$ estimate. We discuss these approaches below.

\begin{rmk}
The smoothness of $u$ up to $\partial \Omega$ was later obtained by Caffarelli-Nirenberg-Spruck via boundary $C^2$ estimates \cite{CNS}. We discuss these estimates in the next section.
\end{rmk}

\begin{rmk}
This result in fact implies that any Alexandrov solution to (\ref{Equation}) is smooth in the (open) set of strict convexity of the solution, and that the agreement set of the graph 
with any supporting hyperplane contains no interior extremal points.
\end{rmk}

\noindent For the remainder of the section we assume that $\partial \Omega$ is smooth and uniformly convex, and $u$ is the unique Alexandrov solution to (\ref{Equation}) in $\Omega$ with zero boundary data.

\subsubsection{The Approach of Cheng-Yau}
Cheng-Yau used a geometric approach based on their (previous) solution of the Minkowski problem \cite{CY2}. In the Minkowski problem we are given a positive smooth function $K$ on $S^{n-1}$, and we construct
a smooth convex body whose boundary has outer unit normal $\nu$, and prescribed Gauss curvature $K(\nu)$. Here we will assume the smooth solvability of this problem, and describe how it is used in \cite{CY}.

\vspace{3mm}

The key observation is that if $\det D^2w = f$ in $\Omega$, then the graph of the Legendre transform of $w$, with downward unit normal, has Gauss curvature
$$K\left(\frac{(x,\,-1)}{(1+|x|^2)^{1/2}}\right) = \frac{1}{f(x)}(1 + |x|^2)^{-\frac{n+2}{2}}$$
for $x \in \Omega$. (For the definition and properties of the Legendre transform, see e.g. \cite{F}). 
It is thus natural to use the solution to the Minkowski problem for the Legendre transform of $u$ to prove regularity.
The obstruction to using this approach directly is the possible existence of a line segment in the graph of $u$ that reaches
$\partial \Omega$ (and thus gets mapped by $\partial u$ to the boundary of the domain for the Legendre transform). 

\vspace{3mm}

To overcome this difficulty, Cheng-Yau solve an approximating problem with right side that blows up near $\partial \Omega$ so that the domain of the Legendre transform is all of 
$\mathbb{R}^n$. More precisely, they construct (Alexandrov) solutions to $\det D^2w = f$ in $\Omega$ with $w|_{\partial \Omega} = 0$, where $f$ is smooth and blows up like distance from the boundary of $\Omega$ to a power smaller than $-1$. Using barriers they show that $\nabla w(\Omega) = \mathbb{R}^n$, which prevents line segments from extending to $\partial \Omega$. They then consider the convex bodies in $\mathbb{R}^{n+1}$ 
bounded by the graph of the Legendre transform of $w$ and large spheres, which may be singular where the spheres intersect the graph. Their Gauss curvatures at $(x,\,-1)/(1+|x|^2)^{1/2}$ agree with $(1/f)(x)(1+|x|^2)^{-\frac{n+2}{2}}$ for
$x \in \Omega' \subset \subset \Omega$. 
By approximating the Gauss curvature measures with positive smooth functions and solving the Minkowski problem, they produce smooth functions that locally solve $\det D^2v = f$ and approximate $w$ locally uniformly.
The Pogorelov and Calabi estimates (adapted to the case of general smooth right side) imply that $w$ is smooth in $\Omega$.

\vspace{3mm}

To produce classical solutions to (\ref{Equation}), they take a limit of solutions with right hand sides $f$ that equal $1$ away from decreasing neighborhoods of $\partial \Omega$, and blow up appropriately fast near $\partial \Omega$.

\subsubsection{The Approach of Lions}
Lions developed a method based on PDE techniques.
Write $\Omega = \{w < 0\}$ for some smooth, uniformly convex function $w$ on $\mathbb{R}^n$. Let $\rho$ be a smooth function that is 
$0$ in $\Omega,$ positive in $\mathbb{R}^n \backslash \overline{\Omega}$, and equals $1$ outside a neighborhood of $\Omega$. Lions solves, for each $\epsilon > 0$, the approximating problem
\begin{equation}\label{Penalization}
\det (D^2u^{\epsilon} - \epsilon^{-1}\rho u^{\epsilon}I) = 1
\end{equation}
in $\mathbb{R}^n$,
where $u^{\epsilon} \in C^{\infty}_b(\mathbb{R}^n)$ (that is, $\|u^{\epsilon}\|_{C^k(\mathbb{R}^n)} < \infty$ for all $k$) and $D^2u^{\epsilon} > \epsilon^{-1}\rho u^{\epsilon} I.$ The
extra term in the equation (\ref{Penalization}) has a favorable sign for the maximum principle, which ``compactifies" the problem and pushes $u^{\epsilon}$ to $0$ outside of $\Omega$ as $\epsilon \rightarrow 0$.

\vspace{3mm}

To see how this works, assume that we have solved the approximating problems (\ref{Penalization}). For any $\delta > 0$, let $\tilde{w}$ be a compactly supported smooth function that agrees with $w$ in $\{w < \delta\}$ and satisfies 
$|\tilde{w}| < 2\delta$ outside of $\Omega$. 
Then $C(\tilde{w} - 4\delta)$ is a subsolution to (\ref{Penalization}) for all $\epsilon$ sufficiently small (depending on $\delta$), with $C$ large independent of $\epsilon,\,\delta$. Since $0$ is a supersolution,
it follows from the maximum principle that $-6C\delta \leq u^{\epsilon} < 0$ outside of $\Omega$ for all $\epsilon$ small.
(It is important that $u^{\epsilon} \in C^{\infty}_b(\mathbb{R}^n$), so that we can use the maximum principle on all of $\mathbb{R}^n$.) Since $u^{\epsilon}$ solve (\ref{Equation}) in $\Omega$, the maximum principle implies that the family
$\{u^{\epsilon}\}$ is Cauchy in $C^0(\mathbb{R}^n)$, and thus converges uniformly as $\epsilon \rightarrow 0$. The Pogorelov estimate implies local smooth convergence in $\Omega$
to the solution of (\ref{Equation}) with $0$ boundary data.

\vspace{3mm}

To solve (\ref{Penalization}), it suffices to obtain global $C^2$ estimates (depending on $\epsilon$) for solutions in $C^{\infty}_b(\mathbb{R}^n)$. Then the Evans-Krylov theorem gives
global $C^{2,\,\alpha}$ estimates, and the method of continuity (see e.g. \cite{GT}, Chapter $17$) can be applied.

For the $C^2$ estimate, the key points are the favorable sign of the extra term, and the concavity of the equation. To see how this works, assume for simplicity that we already have a global $C^1$ estimate.
Differentiating (\ref{Penalization}) twice gives
$$A^{ij}u^{\epsilon}_{eeij} - \frac{\rho}{\epsilon}u^{\epsilon}_{ee}\sum_{i = 1}^n A^{ii} \geq -C(\epsilon) \sum_{i = 1}^n A^{ii},$$
where $A_{ij} = u^{\epsilon}_{ij} - \epsilon^{-1}\rho u^{\epsilon}\delta_{ij}$. 
We have dropped a positive term quadratic in third derivatives on the right side, that comes from the concavity of the equation. 
Let $\overline{w}$ be a smooth compactly supported function that agrees with $w$ in a neighborhood of $\overline{\Omega}$. Then $H - K\overline{w}$ is a super-solution 
to the equation for $u^{\epsilon}_{ee}$ for $H$ and $K$ large, giving an upper bound for $D^2u^{\epsilon}$. The condition $D^2u^{\epsilon} > \epsilon^{-1}\rho u^{\epsilon} I$ gives the required lower bound.

\begin{rmk}
The approach of Lions shows for a large class of nonlinear concave equations that any (viscosity) solution can be uniformly approximated by smooth solutions.
If in addition one has interior $C^2$ estimates, this approach yields smooth solutions. It is useful when boundary $C^2$ estimates are not available
(see e.g. \cite{MS} for an interesting example).
\end{rmk}

\newpage
\section{Boundary Regularity}
In this section we discuss the regularity of solutions to (\ref{Equation}) up to the boundary (due to Caffarelli-Nirenberg-Spruck \cite{CNS}), provided the data are sufficiently smooth.

\subsection{The Boundary Estimates of Caffarelli-Nirenberg-Spruck}
We will discuss following theorem, which is a combination of contributions due to Caffarelli-Nirenberg-Spruck \cite{CNS} and Wang \cite{W}.
\begin{thm}\label{ClassicalDP}
Let $\varphi \in C^3(\partial B_1)$, and let $\omega$ denote the modulus of continuity of $D^3\varphi$. Then there exists a unique solution in $C^{2,\,\alpha}(\overline{B_1})$ of 
$$\begin{cases}
\det D^2u = 1 \text{ in } B_1, \\
u|_{\partial B_1} = \varphi.
\end{cases} $$
Furthermore, the solution satisfies the estimate
\begin{equation}\label{GlobalEstimate}
\|u\|_{C^{2,\,\alpha}(\overline{B_1})} \leq C(n,\, \|\varphi\|_{C^3(\partial B_1)},\, \omega).
\end{equation}
\end{thm}
\noindent An interesting feature of Theorem \ref{ClassicalDP} is the requirement of $C^3$ boundary data to obtain $C^{2,\,\alpha}$ solutions,  in contrast with the uniformly elliptic case. We show in the next section
that this hypothesis is in fact necessary.

\begin{proof}[{\bf Sketch of Proof:}]
Standard (by now) techniques in PDE reduce the problem to obtaining the estimate
$$\|D^2u\|_{C^0(\overline{B_1})} \leq C(n,\, \|\varphi\|_{C^3(\partial B_1)},\, \omega),$$
for a solution $u \in C^{2,\,\alpha}(\overline{B_1})$. Concavity further reduces the problem to bounding the second derivatives on $\partial B_1$. For the details of these reductions, see e.g. \cite{F}.

\vspace{3mm}

\begin{rmk}
The point is that boundary $C^2$ estimates make the problem uniformly elliptic, so we can rely on the general theory for concave, uniformly elliptic equations. 

A historical remark is that this theory was not fully developed at the time of \cite{CNS}.  
In particular, Krylov's boundary Harnack inequality, which is used to obtain global $C^{2,\,\alpha}$ estimates for such equations (see e.g. \cite{CC}, Chapter $9$), was not available.
In \cite{CNS} the authors instead rely on one-sided estimates for some third derivatives at the boundary which are special to the Monge-Amp\`{e}re equation.
\end{rmk}

\vspace{3mm}

Let $T$ be the tangent plane to $\partial B_1$ at the point $\nu$. Then the relation
$$D_T^2u = D_{S^{n-1}}^2\varphi + u_{\nu} I$$
gives a bound on the tangential second derivatives in terms of $\|\varphi\|_{C^2(\partial B_1)}$. (Here we use that $u_{\nu}$ is bounded in terms of $\|\varphi\|_{C^2(\partial B_1)}$, which can be
proved with simple barriers).

\vspace{3mm}

Let $\tau$ denote a tangential vector field on $\partial B_1$ generated by a rotation. Since (\ref{Equation}) is rotation-invariant, $u_{\tau}$ solves the linearized equation $u^{ij}(u_{\tau})_{ij} = 0$, with boundary
data $\varphi_{\tau}$. We can bound the normal derivatives of $u_{\tau}$ on $\partial B_1$ in terms of $\|\varphi\|_{C^3(\partial B_1)}$ using linear functions as barriers. (That is, if $l$ is the linear part
of $\varphi_{\tau}$ at $\nu \in \partial B_1$, then $l(x) \pm C(x \cdot \nu - 1)$ trap the boundary data $\varphi_{\tau}$ for $C$ large depending on $\|\varphi\|_{C^3(\partial B_1)}$ and solve the linearized equation).
This gives the desired bound on the mixed second derivatives $u_{\tau\nu}$.

\vspace{3mm}

Finally, we consider the normal second derivatives. To emphasize ideas, we assume $n = 2$. (The case $n > 2$ is similar). 
After translating and subtracting a linear function, we may assume that we are working in $B_1(e_2)$, and that $u(0) = 0,\, \nabla u(0) = 0$. In particular, $u \geq 0$ by convexity.
Write $\partial B_1$ locally as a graph $x_2 = \psi(x_1)$. By the equation and the previous estimates, it suffices to obtain a positive lower bound for $u_{11}(0)$.
By the $C^3$ regularity of the boundary data, we have the expansion
$$u(x_1,\,\psi(x_1)) = u_{11}(0)x_1^2 + \gamma x_1^3 + o(|x_1|^3).$$
Assume by way of contradiction that $u_{11}(0) = 0$. The key observation is that this implies $\gamma = 0$. Indeed, if not then $u < 0$ somewhere on $\partial B_1$.
We conclude that the set $\{u < h\}$ contains a box $Q$ centered on the $x_2$ axis of length $R(h)h^{1/3}$ and height $R^2(h)h^{2/3}$, with $R(h) \rightarrow \infty$ as $h \rightarrow 0$. It is easy
to construct a convex quadratic polynomial $P$ that is larger than $h$ on $\partial Q$, vanishes at the center of $Q$, and satisfies
$$\det D^2P = 4R(h)^{-6}.$$
For $h$ small this contradicts the maximum principle.
\end{proof}

\vspace{3mm}

\begin{rmk}
If $\Omega$ is smooth and uniformly convex, $u|_{\partial \Omega} = 0$, and $\partial \Omega$ has second fundamental form $II$, then the relation 
$$D_T^2 u = II \, u_{\nu}$$
quickly gives the positive lower bound for the tangential Hessian (the most difficult part of the nonzero boundary data case), see e.g. \cite{F}. 
Here $T$ is a tangent plane to $\partial \Omega$, and $\nu$ is the outer normal at the tangent point.
\end{rmk}


\subsection{Necessity of $C^3$ Data}
Interestingly, $C^3$ boundary data are necessary to obtain solutions to (\ref{Equation}) that are $C^2$ up to the boundary. This was observed by Wang \cite{W}. 

\vspace{3mm}

The point is the affine degeneracy of the equation, which allows us to find solutions with different homogeneities in different directions. 
Indeed, guided by the affine invariance we seek solutions in $\{x_2 > x_1^2\} \subset \mathbb{R}^2$ with the homogeneity
$$u(x_1,\,x_2) = \lambda^{-3}u(\lambda x_1,\, \lambda^2 x_2).$$
Such solutions have cubic growth along the parabolas $x_2 = Cx_1^2$, and satisfy $u_{22} \sim x_2^{-1/2}$ near the origin.

\vspace{3mm}

Take $h(t) = u(t,\,1)$, so that $u(x_1,\,x_2) = x_2^{3/2}h\left(x_2^{-1/2}x_1\right)$. The equation (\ref{Equation}) for $u$ is equivalent to the ODE
\begin{equation}\label{C3ODE}
\frac{1}{4}h''(3h + th') - h'^2 = 1
\end{equation}
for $h$. There exists a positive even, convex solution to (\ref{C3ODE}) on $[-1,\,1]$ (one can e.g. solve (\ref{C3ODE}) a neighborhood of $0$ with the initial conditions $h(0) = 1,\, h'(0) = 0$ and then use the scaling invariance
$h \rightarrow \lambda^{-1}h(\lambda x)$ of the ODE). This defines a solution $u$ to (\ref{Equation}) in the domain $\{x_2 > x_1^2\}$. 
Along the boundary we have $u(x_1,\, x_1^2) = h(1)|x_1|^3$, which is $C^{2,\,1}$ but not $C^3$. Since $u_{22}$ blows up near $0$, this completes the example.


\newpage
\section{Recent Advances and Open Questions}
The Monge-Amp\`{e}re equation remains an active field of study. In this section we list some of the recent results and open questions. We are far from complete.

\subsection{Nondegenerate Case}
Here we discuss the ``non-degenerate" case $\det D^2u = f,$ where $0 < \lambda \leq f \leq \lambda^{-1}$ for some $\lambda > 0$. We call it ``nondegenerate" since, at any point,
$D^2u$ is uniformly elliptic after an affine transformation of determinant $1$.

\begin{enumerate}
\item For many applications (e.g. optimal transport), the right hand side depends on $\nabla u$. Since we don't know a priori anything about the regularity of $\nabla u$ for such equations, it is natural
to consider the case of merely bounded right hand side.

Caffarelli showed that strictly convex solutions to $0 < \lambda \leq \det D^2u \leq \lambda^{-1}$ are locally $C^{1,\,\alpha}$ for some $\alpha > 0$ \cite{Ca1}. For the applications, this reduced the problem to considering $f \in C^{\alpha}$.
He then developed a perturbation theory from the case $f = 1$, showing in particular that if $f \in C^{\alpha}$ then strictly convex solutions are locally $C^{2,\,\alpha}$ \cite{Ca2}.

\item As we have seen, Alexandrov solutions to $\det D^2u = 1$ are smooth on the open set $B_1 \backslash \Sigma$ of strict convexity for $u$. It is natural to ask about the size and structure
of the singular set $\Sigma$. In \cite{M1} we showed that $\mathcal{H}^{n-1}(\Sigma) = 0$, and we constructed examples to show that this cannot be improved. (We in fact
show that if $\det D^2u \geq 1$, then $u$ is strictly convex away from a set of vanishing $\mathcal{H}^{n-1}$ measure.)

One interesting consequence is unique continuation: if $\det D^2u = \det D^2v = 1$ in a connected domain $\Omega$ and the interior of $\{u = v\}$ is non-empty, then $u \equiv v$ in $\Omega$. 
Heuristically, any two points of strict convexity communicate through a path of such points. We used that $u-v$ solves a linear equation with smooth coefficients in the set of strict convexity;
unique continuation when we just have $\lambda \leq f \leq \lambda^{-1}$ remains open.

\item Savin recently developed a boundary version of Caffarelli's interior theory using new techniques \cite{S2}. This led to global versions of Caffarelli's results.

\item De Philippis and Figalli proved interior $L\log^k L$ estimates for the second derivatives of strictly convex solutions to $0 < \lambda \leq \det D^2u \leq \lambda^{-1}$, for any $k$ \cite{DF}. Together with Savin they improved this to $W^{2,\,1+\epsilon}$ regularity, where $\epsilon(n,\,\lambda)$ \cite{DFS}. An interesting consequence of this result is the long-time existence of weak solutions to the semigeostrophic system (see \cite{ACDF1}, \cite{ACDF2}). 
We remark that interior $W^{2,\,1}$ regularity without any strict convexity hypotheses is true, by combining these results with the partial regularity from \cite{M1}. 
\end{enumerate}

\subsection{Degenerate Case}
To conclude we discuss the degenerate case $\det D^2u = f$, where $f = 0$ or $\infty$ at some points. At such points, no affine rescaling of determinant $1$ makes $D^2u$ uniformly elliptic.

\begin{enumerate}
\item To obtain physically meaningful results for the semigeostrophic system, it is natural to investigate $W^{2,\,1}$ regularity in the case $0 \leq f \leq \Lambda < \infty,\, u|_{\partial \Omega} = 0$. 
There are simple counterexamples to $W^{2,\,1}$ regularity in dimension $n \geq 3$
whose graphs have ``corners" on part of a hyperplane. On the other hand, solutions are $C^1$ when $n = 2$ \cite{A}. This gave hope for $W^{2,\,1}$ regularity in $2D$. 
In \cite{M2} we show this is false by constructing solutions in $2D$ whose second derivatives have nontrivial Cantor part.

The support of $f$ in our example is rough (the boundary has ``fractal" geometry), and the example is not strictly convex. The $W^{2,\,1}$ regularity for strictly convex solutions to $0 \leq \det D^2u \leq \Lambda$ in $2D$ remains open.

\item Motivated by applications to the Weyl problem, Daskalopoulos-Savin investigated degenerate Monge-Amp\`{e}re equations of the form $\det D^2u = |x|^{\alpha}$ in $B_1 \subset \mathbb{R}^n$, where $\alpha > 0$ \cite{DaS}. In the case $n = 2$
they characterize the behavior of solutions near $0$. An important tool in the analysis is the partial Legendre transform, which in $2D$ takes the Monge-Amp\`{e}re equation to a linear equation. 
It remains an interesting open problem to analyze the behavior of solutions near $0$ in higher dimensions.

\item Savin recently characterized the boundary behavior of solutions when $f \sim$
distance from $\partial \Omega$ to a positive power, and at boundary points the data separates from the tangent plane of the solution quadratically \cite{S1}. Using these results Le-Savin prove global regularity
of the eigenfunctions of $\det^{1/n}$ in \cite{LeS}.
\end{enumerate}



\newpage

\begin{thebibliography}{10}
\bibitem[A]{A} Alexandrov, A. D. Smoothness of the convex surface of bounded Gaussian curvature. {\it C. R. (Doklady) Acad. Sci. URSS (N. S.)} {\bf 36} (1942) 195-199.
\bibitem[ACDF1]{ACDF1} Ambrosio, L.; Colombo, M.; De Philippis G.; Figalli, A. A global existence result for the semigeostrophic equations in three dimensional convex domains. {\it Discrete Contin. Dyn. Syst.} {\bf 34} (2014), no. 4, 1251-1268.
\bibitem[ACDF2]{ACDF2} Ambrosio, L.; Colombo, M.; De Philippis, G.; Figalli, A. Existence of Eulerian solutions to the semigeostrophic equations in physical space: the 2-dimensional periodic case. 
{\it Comm. Partial Differential Equations} {\bf 37} (2012), no. 12, 2209-2227.
\bibitem[Br]{Br} Brenier, Y. Polar factorization and monotone rearrangement of vector-valued functions. {\it Comm. Pure Appl. Math.} {\bf 44} (1991), no. 4, 375-417.
\bibitem[Ca1]{Ca1} Caffarelli, L. A localization property of viscosity solutions to the Monge-Amp\`{e}re equation and their strict convexity. {\it Ann. of Math.} {\bf 131} (1990), 129-134.
\bibitem[Ca2]{Ca2} Caffarelli, L. Interior $W^{2,p}$ estimates for solutions of Monge-Amp\`{e}re equation. {\it Ann. of Math.} {\bf 131} (1990), 135-150.
\bibitem[CC]{CC} Caffarelli, L.;  Cabr\'{e}, X. {\it Fully nonlinear elliptic equations}, vol. 43 of {\it Amer. Math. Soc. Colloq. Publ.} American Mathematical Society, 1995.
\bibitem[CNS]{CNS} Caffarelli, L.; Nirenberg, L.; Spruck, J. The Dirichlet problem for nonlinear second order elliptic equations I. Monge-Amp\`ere equation. {\it Comm. Pure Appl. Math.} {\bf 37} (1984), 369-402.
\bibitem[Cal]{Cal} Calabi, E. Improper affine hyperspheres of convex type and a generalization of a theorem by K. J\"{o}rgens. {\it Michigan Math. J.} {\bf 5} (1958), 105-126.
\bibitem[CY]{CY} Cheng S.Y.; Yau S.-T. On the regularity of the Monge-Amp\`ere equation $\det \partial^2u/\partial x_i\partial x_j=F(x,u)$. {\it Comm. Pure Appl. Math.} {\bf 30} (1977), no. 1, 41-68.
\bibitem[CY2]{CY2} Cheng S.Y.; Yau S.-T. On the regularity of the solution of the $n$-dimensional Minkowski problem. {\it Comm. Pure Appl. Math.} {\bf 29} (1976), no. 5, 495-516.
\bibitem[CM]{CM} Collins, Tristan C.; Mooney, C. Dimension of the minimum set for the real and complex Monge-Amp\`{e}re equations in critical Sobolev spaces. {\it Anal. PDE} {\bf 10} (2017), no. 8, 2031-2041.
\bibitem[DaS]{DaS} Daskalopoulos, P; Savin, O. On Monge-Amp\`{e}re equations with homogeneous right-hand sides. {\it Comm. Pure Appl. Math.} {\bf 62} (2009), no. 5, 639-676.
\bibitem[DF]{DF} De Philippis, G.; Figalli, A. $W^{2,1}$ regularity for solutions of the Monge-Amp\`ere equation. {\it Invent. Math.} {\bf 192} (2013), 55-69.
\bibitem[DFS]{DFS} De Philippis, G.; Figalli, A.; Savin, O. A note on interior $W^{2,1+\epsilon}$ estimates for the Monge-Amp\`ere equation. {\it Math. Ann.} {\bf 357} (2013), 11-22.
\bibitem[E]{E} Evans, L. C. Classical solutions of convex, fully nonlinear, second-order elliptic equations. {\it Comm. Pure Appl. Math.} {\bf 25} (1982), 333-363.
\bibitem[F]{F} Figalli, A. {\it The Monge-Amp\`{e}re Equation and its Applications}. Z\"{u}rich Lectures in Advanced Mathematics. European Mathematical Society (EMS), Z\"{u}rich, 2017.
\bibitem[GT]{GT} Gilbarg, D.; Trudinger, N. {\it Elliptic Partial Differential Equations of Second Order}. Springer-Verlag, Berlin-Heidelberg-New York-Tokyo, 1983.
\bibitem[Gut]{Gut} Gutierrez, C. {\it The Monge-Amp\`{e}re Equation}. Progress in Nonlinear Differential Equations and their Applications 44, Birkh\"{a}user Boston, Inc., Boston, MA, 2001.
\bibitem[Kr]{Kr} Krylov, N. V. Boundedly nonhomogeneous elliptic and parabolic equations in a domain. {\it Izv. Akad. Nak. SSSR Ser. Mat.} {\bf 47} (1983), 75-108. 
English transl. in {\it Math. USSR Izv.} {\bf 22} (1984), 67-97.
\bibitem[LeS]{LeS} Le, N.; Savin, O. Schauder estimates for degenerate Monge-Amp\`{e}re equations and smoothness of the eigenfunctions. {\it Invent. Math.} {\bf 207} (2017), no. 1, 389-423. 
\bibitem[L]{L} Lions, P. L. Sur les \'equations de Monge-Amp\`ere. I. {\it Manuscripta Math.} {\bf 41} (1983), no. 1-3, 1-43.
\bibitem[M1]{M1} Mooney, C. Partial regularity for singular solutions to the Monge-Amp\`{e}re equation. {\it Comm. Pure Appl. Math.} {\bf 68} (2015), 1066-1084.
\bibitem[M2]{M2} Mooney, C. Some counterexamples to Sobolev regularity for degenerate Monge-Amp\`{e}re equations. {\it Anal. PDE} {\bf 9} (2016), no. 4, 881-891.
\bibitem[M3]{M3} Mooney, C. Thesis, available at {\em https://people.math.ethz.ch/$\sim$mooneyc/}.
\bibitem[MS]{MS} Mooney, C; Savin, O. Regularity results for the equation $u_{11}u_{22} = 1$. Preprint 2018, arXiv:1806.05532. Submitted.
\bibitem[Pog]{Pog} Pogorelov, A. V. {\it The Minkowski multidimensional problem.} Translated from the Russian by Vladimir Oliker, Introduction by Louis Nirenberg, Scripta Series in Mathematics, V. H. Winston \& Sons, 
Washington, D.C.; Halsted Press [John Wiley \&\ Sons], New York-Toronto-London,1978.
\bibitem[S1]{S1} Savin, O. A localization theorem and boundary regularity for a class of degenerate Monge-Amp\`{e}re equations. {\it J. Differential Equations} {\bf 256} (2014), no. 2, 327-388. 
\bibitem[S2]{S2} Savin, O. Pointwise $C^{2,\,\alpha}$ estimates at the boundary for the Monge-Amp\`{e}re equation. {\it J. Amer. Math. Soc.} {\bf 26} (2013), no. 1, 63-99.
\bibitem[U]{U} Urbas, J. Regularity of generalized solutions of Monge-Amp\`{e}re equations. {\it Math. Z.} {\bf 197} (1988), 365-393.
\bibitem[W]{W} Wang, X. J. Regularity for Monge-Amp\`{e}re equation near the boundary. {\it Analysis} {\bf 16} (1996), 101-107.
\end{thebibliography}
\end{document}